\documentclass[12pt,reqno]{amsart} 

\usepackage{amssymb,bbm,enumitem,marginnote,url}
\usepackage{aliascnt}
\usepackage{doi}
\usepackage{pgfplots} 

\usepackage[margin=1.25in]{geometry}

\usepackage{psfrag}

\newcommand{\arxiv}[1]{%
 \href{https://arxiv.org/pdf/#1.pdf}{ArXiv:#1}}

\usepackage{color}

\newcommand\blue[1]{\textcolor{blue}{#1}}

\newtheorem{theorem}{Theorem}

\newaliascnt{lemma}{theorem}
\newtheorem{lemma}[lemma]{Lemma}
\aliascntresetthe{lemma}

\newaliascnt{proposition}{theorem}
\newtheorem{proposition}[proposition]{Proposition}
\aliascntresetthe{proposition}

\newaliascnt{corollary}{theorem}

\aliascntresetthe{corollary}

\newaliascnt{conjecture}{theorem}

\aliascntresetthe{conjecture}

\newaliascnt{openQ}{theorem}

\aliascntresetthe{openQ}

\newaliascnt{quest}{theorem}

\aliascntresetthe{quest}

\newaliascnt{questx}{conjx}

\aliascntresetthe{questx}

\theoremstyle{definition}

\newaliascnt{defn}{theorem}

\aliascntresetthe{defn}

\newaliascnt{example}{theorem}

\aliascntresetthe{example}

\newaliascnt{rem}{theorem}
\newtheorem{rem}[rem]{Remark}
\aliascntresetthe{rem}

\makeatletter
\def\tagform@#1{\maketag@@@{\ignorespaces#1\unskip\@@italiccorr}}
\let\orgtheequation\theequation
\def\theequation{(\orgtheequation)}
\makeatother
\def\equationautorefname~{}

\newcommand{\Vol}{\operatorname{Vol}}
\renewcommand{\deg}{\operatorname{deg}}

\newcommand{\B}{{\mathbb B}}

\newcommand{\R}{{\mathbb R}}

\let\oldmarginnote\marginnote
\renewcommand{\marginnote}[1]{\oldmarginnote{\tiny \blue{#1}}}

\begin{document}
\title[Second Laplacian eigenvalue on sphere]{Maximization of the second Laplacian eigenvalue on the sphere}


	\begin{abstract}
We prove a sharp isoperimetric inequality for the second nonzero eigenvalue of the Laplacian on $S^m$. For $S^{2}$, the second nonzero eigenvalue becomes maximal as the surface degenerates to two disjoint spheres, by a result of Nadirashvili for which Petrides later gave another proof. For higher dimensional spheres, the analogous upper bound was conjectured by Girouard, Nadirashvili and Polterovich. Our method to confirm the conjecture builds on Petrides' work and recent developments on the hyperbolic center of mass and provides also a simpler proof for $S^2$. 
	\end{abstract}
	
\author[]{Hanna N. Kim}
\address{Department of Mathematics, University of Illinois, Urbana--Champaign, IL 61801, U.S.A.}
\email{nekim2@illinois.edu}

	\maketitle 
	

\section{\bf Introduction and results}
Let $g$ be a Riemannian metric on $S^{m}$, $m \geq 2$. The Laplacian operator $\Delta_{g}$ on the manifold $(S^{m},g)$  has a discrete sequence of eigenvalues
\begin{align*}
0=\lambda_{0}(S^m, g) \leq \lambda_{1}(S^m, g) \leq \lambda_{2}(S^m, g) \cdots \leq \lambda_{k}(S^m, g) \leq \cdots \to \infty.
\end{align*}
We normalize the eigenvalues as $\lambda_{k}(S^m, g) \Vol(S^m,g)^{2/m}$ where $\Vol(S^m,g)$ denotes the volume of $S^m$ with respect to the metric $g$ and investigate isoperimetric problems concerning these normalized eigenvalue quantities.

\subsection{\bf 2-dimensional sphere}  
From the work of Hersch \cite{H70}, it is known that the first nonzero eigenvalue on $S^{2}$ has a sharp upper bound:
\begin{equation*}
\lambda_{1}(S^2,g) \Vol(S^2,g)  \leq 8 \pi,
\end{equation*}
with equality if $g$ is conformally equivalent to the standard ``round" metric.

The natural question to ask next is which metric induces the maximum of the second nonzero eigenvalue.

\begin{theorem}[Nadirashvili \protect{\cite{N02}}, Petrides \cite{P14}]\label{thm2dim}
 If $g$ is a Riemannian metric on $S^2$, then
\begin{equation}
\label{secondeigen}
\lambda_{2}(S^2,g) \Vol(S^2,g)  < 16 \pi.
\end{equation}
\end{theorem}
The upper bound of this inequality \autoref{secondeigen} is attained by a sequence of metrics degenerating to a union of two identical round spheres. 

This paper includes a proof of \autoref{thm2dim} that relies on the trial functions constructed by Nadirashvili \cite{N02}, Girouard, Nadirashvili and Polterovich \cite{GNP09}, and Petrides \cite{P14}. The earlier authors essentially built a single trial function satisfying a certain “maximizing direction” property. Instead, in this paper, we construct three valid trial functions and average their Rayleigh quotients. The advantage of this averaging method becomes clear in dimension $m \geq 3$ (see \autoref{newthmhighdim} below) where we get a stronger result than previous authors. Further, our method needs only spherical caps larger than the hemisphere, whereas all caps were considered in earlier works. These techniques are based on the ideas of Freitas and Laugesen \cite{FL20} for domains in hyperbolic space.

\subsection{\bf Higher dimensional spheres} 
Let us denote by
$$\sigma_{m}=\frac{2 \pi^{(m+1)/2}}{\Gamma \left(\frac{m+1}{2} \right)}$$ 
 the volume of the $m$-dimensional unit sphere when $m \geq 2$. The higher dimensional analogue of Hersch's result was shown by El Soufi and Ilias \cite{EI86} who proved that for any metric $g$ conformally equivalent to the standard metric,
\[
 \lambda_{1}(S^{m},g) \Vol(S^{m},g)^{2/m} \leq m \sigma_{m}^{2/m},
\]
where the upper bound is achieved by the round metric. 

For the second eigenvalue, we prove the following upper bound with the right-side attained in the degenerate case of the union of two identical round spheres. This theorem extends \autoref{thm2dim} to all dimensions, confirming a conjecture by Girouard, Nadirashvili and Polterovich \cite[Conjecture 1.2.3]{GNP09}

\begin{theorem}
\label{newthmhighdim}
 If $g$ is a Riemannian metric conformally equivalent to the standard metric on $S^{m}$, $m \geq 2$, then
\begin{equation}
\label{newsecondeigenhigh}
 \lambda_{2}(S^{m},g) \Vol(S^{m},g)^{2/m} <\, m(2 \sigma_{m})^{2/m}.
\end{equation}
\end{theorem}

\begin{rem}
A result by Druet \cite[Theorem 2]{D18} claims that inequality \autoref{newsecondeigenhigh} can go in the opposite direction for some metrics but as \autoref{newthmhighdim} shows, that is not correct.
\end{rem}


Girouard, Nadirashvili and Polterovich \cite{GNP09} in odd dimensions and Petrides \protect{\cite{P14}} in all dimensions established a weaker inequality than \autoref{newthmhighdim}, with an additional constant factor strictly greater than 1 multiplied to the right side of \autoref{newsecondeigenhigh}. By averaging over all $m+1$ valid trial functions before applying H\"{o}lder’s inequality, we avoid the constant factor. This is explained at the end of the proof.

\subsection{\bf Related work}
For developments related to maximizing eigenvalues on surfaces among metrics in a fixed conformal class, we briefly mention the following papers. Petrides \cite[Theorem 1]{P17} showed that under some natural eigenvalue gap conditions, each conformal class contains a maximal metric for the $k$-th area-normalized eigenvalue on a surface. A recent proof using a bubbling-tree construction for harmonic maps was given by Karpukhin, Nadirashvili, Penskoi and Polterovich \cite{KNPP20}. Karpukhin and Stern \cite{KS20} investigated the existence of conformally maximizing metrics on general surfaces using min-max energy problems for sphere-valued maps.

The eigenvalues on $S^2$ have been extensively studied after Hersch's work. For the $k$-th eigenvalue on $S^2$, Kao, Lai and Osting \cite{KLO17} provided numerical support for the conjecture that $k$ disjoint spheres give the maximizers. Nadirashvili and Sire \cite{NS17} confirmed the sphere conjecture for $k=3$. Karpukhin, Nadirashvili, Penskoi and Polterovich \cite[Theorem 1.2]{KNPP19}, by relying importantly on earlier work of Petrides \cite{P17}, proved this conjecture for all $k$: the $k$-th eigenvalue on $S^{2}$ is maximized by a sequence of metrics degenerating to a union of $k$ disjoint identical round spheres.

Maximization of eigenvalues on the projective plane $\R \mathbb{P}^2$ has been investigated by Li and Yau \cite{LY82} for $\lambda_1$, by Nadirashvili and Penskoi \cite{NA18} for $\lambda_2$, and by Karpukhin \cite[Theorem 1.2]{K19} for all $\lambda_k$.

Some relevant problems for optimizing eigenvalues on plane domains with boundary conditions can be found in the following papers. Girouard and Laugesen \cite{GL19} showed that the third eigenvalue of the Robin Laplacian on a simply-connected planar domain of given area is maximal for a disjoint union of two disks. Karpukhin and Stern \cite[Theorem 1.5, and 1.13]{KS20} showed that the first two Steklov eigenvalues satisfy sharp inequalities by using min-max quantities associated to the energy functional from the closed surface to a sphere. Girouard, Karpukhin and Lagac\'{e} \cite[Theorem 1.2]{GKL20} obtained the sharp upper bound for all Steklov eigenvalues in $\R^{2}$ by using the sharp upper bounds for weighted Neumann eigenvalues and homogenizing with a perforated domain to approximate the Neumann eigenvalues with Steklov eigenvalues. 

\section{\bf Proof of Theorem 1 and Theorem 2} 
We will prove the theorems together since \autoref{thm2dim} is simply the 2-dimensional case of \autoref{newthmhighdim}.
\subsection{\bf Overview of proofs}
For the first eigenvalue $\lambda_{1}$, Hersch \cite{H70} and El Soufi and Ilias \cite{EI86} found trial functions orthogonal to the constant by composing the eigenfunctions of the round sphere (which are the $m+1$ coordinate functions) with a M\"{o}bius transformation to move the center of the mass to the origin. For the second eigenvalue $\lambda_{2}$, Nadirashvili \cite{N02} in dimension 2 and then Girouard, Nadirashvili and Polterovich \cite{GNP09} in odd dimensions and Petrides \cite{P14} in all dimensions folded the measure across a spherical cap and rearranged with respect to the center of mass, to get functions that are even with respect to reflection in the cap boundary and orthogonal to the constant. Among these families of functions, they constructed by a topological argument a two dimensional subspace of “maximizing directions”. This subspace contains a trial function that is orthogonal to the constant and also to the first non-constant eigenfunction of the metric and putting this trial function into the Rayleigh quotient finishes the proof.

The main simplification in this paper is that we find $m+1$ trial functions that are orthogonal to the constant and the first eigenfunction. The two-step method for doing this is first to get orthogonality to the constant by choosing a M\"{o}bius transformation, which can be done due to uniqueness of the Hersch center of mass point, and then obtain orthogonality to the first excited state by choosing the spherical cap suitably, with the help of Petrides's degree theory lemma (\autoref{petrides} below). This two-step approach is adapted from the work of Girouard and Laugesen \cite{GL19} for Robin eigenvalues in the plane and Freitas and Laugesen \cite{FL20} for domains in Euclidean and hyperbolic space. An alternative one-step approach for obtaining the orthogonality has been developed by Karpukhin and Stern \cite[Lemma 4.2]{KS20}. As a remark, Hersch and all later authors were relying on the methods of Szeg\H{o} \cite{S54} and Weinberger \cite{W56} for domains in Euclidean space.

\subsection{First eigenfunctions on the sphere} 
Let us briefly review the first eigenfunctions on the unit sphere $S^{m}$ with respect to the standard metric $g_{0}$. The eigenfunctions are the spherical harmonics of degree 1, that is, the coordinate functions $y_{j}$ for $j=1, \dots, m+1$. The eigenvalue $\lambda_{1}(S^{m},g_{0})=m$ has multiplicity $m+1$ since
\begin{equation}
\label{first}
-\Delta_{g_{0}}y_{j}=my_{j}, \qquad j=1 ,\dots, m+1.
\end{equation}

\subsection{M\"{o}bius transformations} 
M\"{o}bius transformation will play an important role in defining the  trial functions. We will write $\B^{m+1}$ for the unit ball in $\R^{m+1}$, so that $S^{m} = \partial \B^{m+1}$. Consider the following M\"{o}bius transformations on the closed ball \cite[eq.(2.1.6)]{S16}, parametrized by $x \in \B^{m+1}$ and defined by 
\[
T_{x}: \overline{\B^{m+1}} \to \overline{\B^{m+1}},
\]
\begin{equation}
\label{mobius}
T_{x}(y)=\frac{(1+2x \cdot y +|y|^{2})x + (1-|x|^{2})y}{1+2x \cdot y +|x|^{2}|y|^{2}}
, \qquad  y \in \overline{\B^{\, m+1}}.
\end{equation}
Note $T_{0}$ becomes just the identity map on the ball. Also, $T_{x}(0)=x$, $T_{-x}=(T_x)^{-1}$, and $T_{x}$ maps $S^{m}$ to $S^{m}$, fixing the points $y= \pm x/|x|$. 

\subsection{Spherical caps, reflection, and folding} 
Next, we describe spherical caps. For any unit vector $p$ on the sphere, define the closed hemisphere
\[
H_{p}=\{y \in S^{m} : y \cdot p \leq 0\}, \qquad p \in S^m.
\]
Then let 

\[
H \equiv H_{p,t} = T_{pt}(H_{p}), \qquad p \in S^m , \qquad t \in [0,1),
\]
be the spherical caps defined as the image of the hemisphere under a M\"{o}bius transformation. (We could also consider negative values of $t$, but only the positive values will be needed in this paper.) Since the M\"{o}bius transformation sends a boundary to a boundary, that is, $T_{pt} (\partial H_{p})=\partial H_{p,t}$, we can write the hyperbolic cap explicitly as
\[
H_{p,t}= \left\{ y \in S^{m} : y \cdot p \leq \frac{2t}{1+t^{2}} \right\},
\]
which can be checked using \autoref{mobius}. Note that as $t$ approaches $1$, the spherical cap $H_{p,t}$ expands toward $p$ and covers almost all the sphere except for $p$ itself.

Define $R_{p}$ to be reflection in the hyperplane through the origin and perpendicular to the unit vector $p$:
\[
R_{p}(y)=y -2 (y \cdot p) p, \qquad y \in S^m.
\]
Note that $R_{p}(p)=-p$, which implies that the reflection map sends $p$ to its antipodal point. We can define a reflection map across the boundary of the general spherical cap $H_{p,t}$ by conjugation, that is, let
\begin{equation}
\label{conjug}
R_{H} \equiv R_{p,t}= T_{pt} \circ R_{p} \circ ( T_{pt})^{-1}: S^{m} \to S^{m}.
\end{equation} 
Here, the inverse of the M\"{o}bius transformation can be written as $ (T_{pt})^{-1}=T_{-pt}$, which can be checked by calculation. Finally, define a ``fold map" that reflects the complement of the spherical cap across the boundary:
\[
F_{H}(y) \equiv F_{p,t}(y)=
\begin{cases}
y , & y \in H , \\
	R_{H}(y) , & y \in S^{m} \setminus H.
\end{cases}
\]
Observe that the map sends the boundary of the spherical cap $\partial H$ to itself. For simplicity, we will exploit the notation $F_{H}$ instead of $F_{p,t}$ when it is clear from the context (similarly for $R_{H}$). 

\subsection{Center of mass}
The concept of center of mass plays an important role in our proof. We say a point $c \in \B^{m+1}$ is the center of mass of a Borel measure $\mu$ on the sphere if it satisfies 
\[
\int_{S^m} T_{-c}(y)  \, d\mu(y) =0.
\]
In other words, after the measure is pushed forward by the M\"{o}bius transformation $T_{-c}$, the center of mass lies at the origin:
\[
\int_{S^m} y   \,  d[(T_{-c})_{*} \mu](y) =0.
\]

\subsection{Trial functions}

Let $Y(y)=y$ be the identity map on the sphere, so that each component $Y_{j}(y)=y_{j}$ is an eigenfunction for the round sphere. 
Construct a map
\begin{equation}
\label{vectorfield}
y \mapsto (Y \circ T_{-c_{H}} \circ F_{H})(y)
\end{equation}
for some point $c_{H} \equiv c_{p,t}$ to be chosen dependent on the spherical cap $H$. The map first folds the sphere onto one side of the spherical cap and then renormalizes with $T_{-c_H}$. We will drop the identity map ``$Y$" in the work that follows since its role in \autoref{vectorfield} is mainly to emphasize that the trial functions, which are the $m+1$ components of \autoref{vectorfield}, are found by precomposing the coordinate functions $y_{j}$ with the M\"{o}bius transformation and the fold.

The $m+1$ trial functions need to be orthogonal to the constant and to the first eigenfunction in order to be valid trial functions for $\lambda_2$. The dimension of our parameter space matches the number of orthogonality conditions, as follows. Each orthogonality requirement gives $m+1$ conditions, making $2m+2$ conditions in total. The parameters $(p,t,c_{H}) \in S^{m} \times [0,1) \times \R^{m+1}$ belong to a space of dimension $2m+2$. Hence, our construction suggests that it is possible to satisfy the orthogonality conditions using the parameters. 

\subsection{Orthogonality of trial functions to the constant.} 
By the definition of the trial functions in \autoref{vectorfield}, for orthogonality to the constant, we want the integral of the vector of this trial functions over $S^{m}$ to be zero, meaning
\begin{equation}
\label{constant}
\int_{S^{m}}^{} (T_{-c_{H}} \circ F_{H})(y) \, dv_{g}=0,
\end{equation}
where $v_{g}$ is the volume measure with respect to the metric $g$. The existence of a center of mass point $c_{H}$ satisfying \autoref{constant} is due to the result of Hersch \cite{H70} applied to the measure $\mu = (F_{H})_{*} v_{g}$, which is the pushforward under the fold map of the volume measure $v_g$. 

We later need $c_{H}$ to be unique, and depend continuously on the parameters of $H$. For those facts we rely on Laugesen \cite[Corollary 5]{L20b}. The hypothesis of that corollary is satisfied because the pushforward measure is a finite Borel measure and $0=\mu(\{y\}) < \frac{1}{2}\mu(S^{m})$ for all $y \in S^{m}$. The conclusion of the corollary gives uniqueness. Moreover, the center of mass $c_{H} = c_{p,t}$ depends continuously on $(p,t) \in S^{m} \times [0,1)$ as now we explain. (We later use this fact to prove the orthogonality to the first excited state.) It is enough to show that the measure $(F_{H})_{*}v_{g} \equiv v_{g,H}$ is weakly continuous on $(p,t)$, because then we can apply Laugesen \cite[Corollary 5]{L20b}.                                                                                                                                                                                                                                                                                                                                                                                                                                                                                                                                                                                                                                                                                                                                                                                                                                                                                                                                                                                                                                                                                                                                                                                                                                                                                                                                                                                                                                                                                                                                                                            Suppose that $t \in [0,1)$ and $p \in S^{m} $, and let $(p_{k},t_{k})$ be a sequence converging to $(p,t)$. Write $H_{k} \equiv H_{p_{k},t_{k}}$. To get weak continuity of the measures, it is enough to show that for any continuous function $\psi$ on $S^{m}$,
\begin{align*}
\int_{S^{m}}^{} \psi \, dv_{g,H_k} \to \int_{S^{m}}^{} \psi \, dv_{g,H} \qquad \text{as } k \to \infty.
\end{align*}
Equivalently, we want $\int_{S^{m}}^{} (\psi \circ F_{H_k}) \, dv_{g} \to \int_{S^{m}}^{} (\psi \circ F_{H})  \, dv_{g}$ as $k \to \infty$. Note that $F_{H_{k}}$ coverges pointwise to $F_{H}$. Since $\psi$ is a continuous function, it is bounded, and so dominated convergence gives the result. 

The behavior of the center of mass point in the limiting case $t \to 1$ will be important for our argument. Let us consider $p_{k} \to p$ and $t_{k} \to 1$. Denote the center of mass with respect to $v_g$ as $c(g)$, which satisfies
\begin{equation}
\label{center}
\int_{S^{m}} T_{-c(g)}(y) \, dv_g=0,
\end{equation} 
by the definition of center of mass. Then, our claim is that the sequence of centers of mass corresponding to $(p_k, t_k)$ converges to that of the whole sphere, that is, 
\begin{equation}
\label{limitcenter}
c_{p_k ,t_k} \to c_{p,1}=c(g).
\end{equation} 
Notice the limiting value does not depend on $p$. To understand why this claim works, we need to extend some of the definitions above. Although the M\"{o}bius transformation and the fold map do not extend continuously when $t=1$, the pushforward measure does have a  continuous extension: from what we observed earlier about the behavior of $H_{p,t}$ as $t$ goes to 1 (it extends to cover all of the sphere except for $p$), the fold map tends pointwise to the identity except at $p$. So $v_{g,H_{k}}$ converges to $v_g$ weakly and we can still apply \cite[Corollary 5]{L20b} to get convergence of the center of mass points. 

\subsection{Orthogonality of trial functions to the first excited state.} 
Let $f$ be a first excited state, meaning an eigenfunction of $\lambda_{1}(S^{m},g)$. We want to show that there exists some spherical cap $H$ such that the vector of trial functions \autoref{vectorfield} is orthogonal to the first excited state $f$. Equivalently, we want the vector field 
\[
V(p,t)= \int_{S^{m}}^{}T_{-c_{H}} (F_{H}(y)) f(y) \, dv_{g}
\]
to vanish at some point $(p,t) \in S^{m} \times [0,1)$, where $H=H_{p,t}$. Using the continuous dependence result in the previous section, the vector field $V(p,t)$ is continuous. We will exploit the following topological result by Petrides \cite{P14}, which was recently given a new proof by Freitas and Laugesen \cite[Theorem 2.1]{FL20}.

\begin{proposition}[Petrides \protect{\cite{P14}{, claim 3}}]\label{petrides}Suppose $m \geq 1$. If $ \phi:S^{m} \to S^{m}$ is continuous and 
\begin{equation}
\label{refsym}
\phi(-p)=R_{p}(\phi(p)), \qquad p \in S^{m},
\end{equation}
then the degree of $\phi$ is nonzero, meaning it is not homotopic to a constant map. In particular, if m is odd then $\deg(\phi)=1$ and if $m$ is even then $\deg(\phi)$ is odd.
\end{proposition}
This theorem implies that a continuous map between spheres with the reflection property \autoref{refsym} has nonzero degree. We will prove the reflection symmetry of our vector field \autoref{vectorfield} when $t=0$ by dividing into a few steps. This part proceeds similarly to Freitas and Laugesen \cite{FL20} in hyperbolic space.

First observe that the relationship between the reflection map and a M\"{o}bius transformation:
\begin{equation}
\label{TandR}
T_{R_{p}x}(R_{p}y)=R_{p}T_{x}(y)
\end{equation}
for any points $p, x, y$. This can be verified by using the definition in \autoref{mobius}. Note that although formula \autoref{TandR} looks somewhat complicated, its Euclidean analogue (with M\"{o}bius transformations replaced by translations) simply says that $R_{p}x+R_{p}y=R_{p}(x+y)$, which is true because $R_p$ is a linear map.

\begin{lemma}[Reflection invariance of the center of the mass] \label{invcenter}
For $p \in S^{m}$,
\[
c_{-p,0}=R_{p}(c_{p,0}).
\]
\end{lemma}
\begin{proof}
It is clear from the definition of the spherical cap that the boundaries of $H_{p,0}$ and $H_{-p,0}$ are equal. The reflection $R_{p}$ maps each of these hemispherical caps to each other, $R_{p}(H_{p,0})=H_{-p,0}$, and vice versa. This implies a simliar relationship between the fold maps, which is,
\begin{equation}
\label{FandR}
F_{-p,0}(y)=R_{p}(F_{p,0}(y)).
\end{equation}

In order to show that the reflection map sends the center of mass $c_{p,0}$ to $c_{-p,0}$, we want to check whether the reflection of $c_{p,0}$ satisfies the condition \autoref{constant} for $H_{-p,0}$. Therefore, it is enough to show that 
\begin{align*}
\int_{S^{m}}^{}T_{-R_{p}(c_{p,0})}(F_{-p,0}(y)) \, dv_{g}=0.
\end{align*}
The left-hand side above is equal to 
\begin{align*}
\int_{S^{m}}^{}T_{-R_{p}(c_{p,0})}(F_{-p,0}(y)) \, dv_{g}
&=\int_{S^{m}}^{}R_{p}(T_{-c_{p,0}}(F_{p,0}(y)))\, dv_{g} \qquad \text{by \autoref{TandR} and \autoref{FandR}}
\\&=R_{p}\int_{S^{m}}^{}T_{-c_{p,0}}(F_{p,0}(y))\, dv_{g} \qquad \text{by linearity of $R_{p}$}
\\&=0,
\end{align*}
by the center of mass condition \autoref{constant} with $t=0$.
\end{proof}
Now, we want to show that our vector field satisfies the reflection symmetry property.
\begin{lemma}[Reflection symmetry when $t=0$]\label{refsymt=0}
For all $p \in S^{m}$
\[
V(-p,0)=R_{p}(V(p,0)).
\]
\end{lemma}

\begin{proof}
We can calculate
\begin{align*}
V(-p,0)&= \int_{S^{m}}T_{-c_{-p,0}}(F_{-p,0}(y))f(y) \, dv_{g} 
\\&=  \int_{S^{m}}T_{-R_{p}c_{p,0}}(R_{p}F_{p,0}(y))f(y) \, dv_{g} \qquad \text{by \autoref{invcenter} and \autoref{FandR}}
\\&= \int_{S^{m}}R_{p}T_{-c_{p,0}}(F_{p,0}(y))f(y) \, dv_{g} \qquad \text{by \autoref{TandR}}
\\ &=R_{p} \int_{S^{m}}T_{-c_{p,0}}(F_{p,0}(y))f(y) \, dv_{g} \qquad \text{by linearity of }R_{p}
\\&=R_{p}(V(p,0)).
\end{align*} 
\end{proof}
We need to observe what happens to the vector field when $t \to 1$. The lemma below shows that the vector field does not depend on $p$ in the limit.
\begin{lemma}[$t \to 1$] \label{indep}
As $p \to  q \in S^{m}$ and $t \to 1$,   
\[
 V(p,t) \to V(g)
\]
where $V(g)=\int_{S^{m}}^{}T_{-c(g)}(y)f(y) \, dv_{g}$ is a constant vector (independent of $q$).
\end{lemma}
\begin{proof}
In the limit, the fold map $F_{H_{p,t}}$ goes to the identity map except at the single point $q$. And $c_{p,t} \to c(g)$ by \autoref{limitcenter}. Thus, it holds that
\[
V(p,t) = \int_{S^{m}}^{}T_{-c_{p,t}}(F_{p,t}(y))f(y)  \, dv_{g} \to \int_{S^{m}}^{} T_{-c(g)} (y)f(y) \, dv_{g}=V(g),
\]
by dominated convergence. 
\end{proof}
By the last lemma, we can regard $V(p,t)$ as extending continuously to $t=1$, by writing $V(p,1)=V(g)$.
Combining the results we have vanishing of the vector field:
\begin{proposition}[Vanishing of the vector field]\label{vanish}
$V(p,t)=0$ for some $p \in S^{m}$ and $t \in [0,1]$.
\end{proposition}

\begin{proof}
Suppose not, that is, $V(p,t) \neq 0 $ for all $(p,t)$.

Consider $t=1$. The map $p \mapsto V(p,1)/|V(p,1)|$ from $S^{m}$ to $S^{m}$ is constant by \autoref{indep}. Hence its topological degree is zero.

On the other hand, when $t=0$, the map $p \mapsto V(p,0)/|V(p,0)|$ satisfies the reflection symmetry condition \autoref{refsym}, by \autoref{refsymt=0}. Hence, the map has nonzero degree by \autoref{petrides}.

Topological degree is a homotopy invariant for continuous maps between spheres. We have a contradiction and therefore, the vector field $V(p,t)$ vanishes at some point.
\end{proof}

\subsection{Rayleigh quotient estimate} 
Finally, we use a variational characterization to obtain the upper bound on the second eigenvalue. Fix the parameters $(p,t)$ at which the vector field $V(p,t)$ vanishes, as in \autoref{vanish}. First suppose $t<1$. The trial functions, which are components of \autoref{vectorfield}, are orthogonal to the constant and the first excited state for the metric $g$, by \autoref{constant} and \autoref{vanish}.

The Rayleigh principle says that
\begin{align}
\label{Rayleigh}
\lambda_{2}(S^{m},g) \leq \frac{\int_{S^m}^{}|\nabla_{\!g\,}u|_{g}^{2} \, dv_{g}}{\int_{S^m}^{}u^{2} \, dv_{g}}
\end{align}
whenever $u$ is a trial function in $H^{1}(S^m)$ that is orthogonal to the constant and the first eigenfunction. We apply this principle to the trial functions $(T_{-c_{H}} \circ F_{H})_{j}$ for $j=1, \dots , m+1$. By substituting them into \autoref{Rayleigh}, we get
\begin{equation}
\label{Rayleigh2}
\lambda_{2}(S^{m},g) \leq \frac{\int_{S^m}^{}|\nabla_{\!g\,}(T_{-c_{H}} \circ F_{H})_{j}|_{g}^{2} \, dv_{g}}{\int_{S^m}^{}(T_{-c_{H}} \circ F_{H})_{j}^{2} \, dv_{g}}, \qquad j=1, \dots , m+1.
\end{equation} 

Multiply by the denominator on both sides of the inequality \autoref{Rayleigh2} and sum for all $j$'s. The sum of the denominators is equal to the surface area of the sphere corresponding to the metric $g$, that is, $\int_{S^m}^{} \sum_j (T_{-c_{H}} \circ F_{H})_{j}^{2} \, dv_{g}=\Vol(S^{m},g)$, since $(T_{-c_H} \circ F_H)(y)$ is a point on the unit sphere and hence has magnitude 1. Hence the inequality becomes 
\[
\lambda_{2}(S^{m},g) \Vol(S^{m},g) \leq  \int_{S^m}^{} \sum_{j=1}^{m+1}|\nabla_{\!g\,}(T_{-c_{H}} \circ F_{H})_{j}|_{g}^{2} \, dv_{g}.
\]
By applying H\"{o}lder's inequality to the the right-hand side, 
\begin{align*}
 &\lambda_{2}(S^{m},g) \Vol(S^{m},g) \\ &\leq \left( \int_{S^m}^{}  \left( \sum_{j=1}^{m+1} |\nabla_{\!g\,}(T_{-c_{H}} \circ F_{H} )_{j}|_{g}^{2}\right)^{\!\! m/2} dv_{g}  \right)^{ \! \!2/m} \! \Vol(S^{m},g)^{1-2/m}.
\end{align*}
(This application of H\"{o}lder is not needed when $m=2$.) We introduce a new notation: for $F:S^{m} \to S^{m}$ we denote $|\nabla_g F|_{g} = \sqrt{\sum_{j=1}^{m+1} |\nabla_{\!g\,}(F )_{j}|_{g}^{2}}$. Then
\begin{align}
 &\lambda_{2}(S^{m},g) \Vol(S^{m},g)^{2/m}  \nonumber
\\&\leq   \left(  \int_{S^m}^{}|\nabla_{\!g\,} (T_{-c_{H}} \circ F_{H})|_{g}^{m} \, dv_{g} \right)^{\! 2/m} \nonumber
\\& = \left( \int_{S^m}^{}|\nabla_{\!g_{0}\,}(T_{-c_{H}} \circ F_{H})|_{g_{0}}^{m} \, dv_{g_0} \right)^{\! 2/m} \nonumber
\\ &\quad \text{by changing $g$ to the round metric $g_{0}$ (conformal invariance, \autoref{confinv2})} \nonumber
\\ &=   \left( \int_{H }^{}|\nabla_{\!g_{0}\,}(T_{-c_{H}})|_{g_{0}}^{m} \, dv_{g_0}+ \int_{R_{H}(H) }^{}|\nabla_{\!g_{0}\,}(T_{-c_{H}} \circ R_{H})|_{g_{0}}^{m} \, dv_{g_0} \right)^{\! 2/m} \label{applyfold}
\end{align}
by splitting into integrals over $H$ and its complement and then using the definition of the fold map. The second integral in \autoref{applyfold} equals the first one, by conformal invariance of the Dirichlet integral (\autoref{confinv1}), since  $R_{H}$ is a composition of M\"{o}bius transformations and a reflection, as defined in \autoref{conjug}. The boundary of the spherical cap is counted twice but can be neglected since it has measure zero. So, we get
\begin{align}
 \lambda_{2}(S^{m},g) \Vol(S^{m},g)^{ 2/m}  \nonumber
 &\leq 2^{\, 2/m}  \left( \int_{H}^{}|\nabla_{\! g_{0}}(T_{-c_{H}})|_{g_{0}}^{m} \, dv_{g_0} \right)^{\! 2/m} \nonumber
\\&=  2^{\, 2/m} \left(  \int_{T_{-c_{H}}(H)}^{}|\nabla_{\! g_{0}} y|_{g_{0}}^{m} \, dv_{g_0} \right)^{\! 2/m}  \label{coor}
\end{align}
by conformal invariance, \autoref{confinv1}, applied with $f(y)=y$ being the identity. Next, bounding the integral in \autoref{coor} with the integral over the whole sphere, we find
\begin{align}
\label{dirichlet}
\int_{T_{-c_{H}}(H)}^{}|\nabla_{\!g_{0}\,} y|_{g_{0}}^{m} \, dv_{g_0} 
&<\int_{S^{m}}^{}|\nabla_{\!g_{0}\,} y|_{g_{0}}^{m} \, dv_{g_0} \nonumber
\\ &=\int_{S^{m}}^{} \left( \sum_{j=1}^{m+1}  |\nabla_{\!g_{0}\,}y_j|_{g_{0}}^2 \right)^{\!m/2} \, dv_{g_0} . 
\end{align}
Here, for each $y \in S^{m}$, we write $y_j = (e_j, y)$ where $e_j$ is a standard unit vector. By projecting the gradient of $y_j$ from the tangent space $T \R^{m+1}$ to $TS^{m}$, that is, projecting $e_j$ onto the sphere, we find
\begin{align}
\label{estimate}
\int_{T_{-c_{H}}(H)}^{}|\nabla_{\!g_{0}\,} y|_{g_{0}}^{m} \, dv_{g_0} 
&< \int_{S^{m}}^{} \left( \sum_{j=1}^{m+1}  | e_j - (e_j, y) y|^{2} \right)^{\! m/2} \, dv_{g_0}  \nonumber
\\ &=\int_{S^{m}}^{} \left( \sum_{j=1}^{m+1}  1 - (e_j, y)^2 \right)^{\! m/2} \, dv_{g_0}  \nonumber
\\ &=\int_{S^{m}}^{} \left( m+1 -\sum_{j=1}^{m+1} y_j^2 \right)^{\! m/2} \, dv_{g_0}  \nonumber 
\\ &=\int_{S^{m}}^{} m^{(m/2)} \, dv_{g_0}  \nonumber \text{ since $y$ is a point on the sphere}
\\ &= \sigma_m m^{(m/2)}. 
\end{align}
After putting this estimate back into \autoref{dirichlet} and then to \autoref{coor}, we obtain \autoref{newsecondeigenhigh}, which finishes the proof. 

For the 2-dimensional case in \autoref{newthmhighdim}, a simpler way of calculating the Dirichlet energy on the unit sphere is to apply Green's theorem to the Dirichlet integral,
\[
\int_{S^{2}}^{} |\nabla_{\!g_{0}\,} y_{j}|_{g_{0}}^{2} \, dv_{g_0} 
=-\int_{S^{2}}^{} y_{j} \, \Delta_{g_{0}\,} y_{j} \, dv_{g_0}
 =\lambda_{1}\int_{S^{2}}^{} y_{j}^{2} \, dv_{g_0},
\]
for $j =1, 2, 3$. Summation over $j$ gives the area $4\pi$ of the sphere times its eigenvalue $\lambda_1=2$  by \autoref{first}. This gives an upper bound of $16 \pi$ for \autoref{coor}, as wanted for \autoref{secondeigen}.

Lastly, let us consider the case when $t=1$. In this case, we define the trial functions $u_{j}(y)=(T_{-c(g)}(y))_j$ for $j =1, \dots, m+1$. Then the orthogonality conditions are satisfied by \autoref{center} and \autoref{indep} (where $V(g)=0$ since we assume that the vector field vanishes at $t=1$). The Rayleigh quotient calculation becomes simpler since we are working on the whole sphere and not on the spherical caps. Hence, the formula before \autoref{applyfold} becomes
\begin{align*}
 \lambda_{2}(S^{m},g) \Vol(S^{m},g)^{2/m} 
&\leq   \left( \int_{S^m }^{}|\nabla_{\!g_{0}\,} T_{-c(g)}|_{g_{0}}^{m} \, dv_{g_0} \right)^{\! 2/m}
\\&=  \left( \int_{S^m }|\nabla_{\!g_{0}\,} y |_{g_{0}}^{m} \, dv_{g_0} \right)^{\! 2/m} \qquad \text{by conformal invariance}
\\&= m \sigma_m^{2/m} <   m (2\sigma_m)^{2/m}
\end{align*}
where we obtained the last equality by using the estimate \autoref{estimate}. 

\appendix\section{Conformal invariance of Dirichlet integral}  
In this paper, two kinds of conformal invariance are used, for the change of metric by a conformal factor and the transplantation of a function by conformal diffeomorphisms. Note that both Lemmas in this appendix are well-established results and the purpose of stating them is to provide convenience for the readers. Recall that $T_{x}$ is a M\"{o}bius transformation on the closed unit ball and $R_{p}$ is reflection through the plane perpendicular to the unit vector $p$.

\begin{lemma}(M\"{o}bius transformation and reflection invariance of Dirichlet integral)\label{confinv1}
If $\Omega$ is a subdomain of $S^{m}$, $m \geq 2$, and each coordinate of $f: \Omega \to \R^{m+1}$ is a continuously differentiable function, then
\begin{align*}
\int_{\Omega}^{} |\nabla_{\!g_{0}} (f \circ T_{x})|_{g_{0}}^{m} \, dv_{g_{0}}=\int_{T_{x}(\Omega)}^{} |\nabla_{\!g_{0}} f|_{g_{0}}^{m} \, dv_{g_{0}}(y)
 \\ \text{and} \qquad 
\int_{\Omega}^{} |\nabla_{\!g_{0}} (f \circ R_{p})|_{g_{0}}^{m} \, dv_{g_{0}}=\int_{R_{p}(\Omega)}^{} |\nabla_{\!g_{0}} f|_{g_{0}}^{m} \, dv_{g_{0}}(y)
\end{align*}
where $\nabla_{\!g_{0}}$ and $v_{g_{0}}$ are the gradient and volume measure for the standard round metric on the sphere, respectively.
\end{lemma}
The lemma follows easily since the M\"{o}bius transformation $T_{x}$ is a conformal map on $\R^{m+1}$ and so its derivative matrix is the product of a positive factor and an orthogonal matrix, at each point.


Also, conformally equivalent metrics give the same Dirichlet energy. Let $[g]$ denote the conformal class of metrics that are equivalent to the metric $g$.
\begin{lemma}(Invariance of Dirichlet integral between conformally equivalent metrics)\label{confinv2}
 If $h \in [g]$ is a Riemannian metric on a manifold $M$ with dimension $m$, then
\begin{align*}
\int_{M}^{} |\nabla_{\!g\,}f|_{g}^{m} \, dv_{g}=\int_{M}^{} |\nabla_{\! h\,}f|_{h}^{m} \, dv_{h}
, \qquad f \in C^{1}(M).
\end{align*}
\end{lemma}

We omit the proof since the lemma follows easily by writing $h=cg$ for each point.

\section*{\textbf{Acknowledgements}}
I am grateful for support from the University of Illinois Campus Research Board award RB19045 (to Richard Laugesen). Mikhail Karpukhin provided helpful feedback on the first draft of the paper, in particular suggesting to take the sum inside the integral before applying  H\"{o}lder's inequality, which turns out to be what El Soufi and Ilias did in their treatment of the first eigenvalue.

\bibliographystyle{plain}


\begin{thebibliography}{99}










\bibitem{D18}
O. Druet. 
\emph{On the second conformal eigenvalue of the standard sphere.}
Asian J. Math. 22, no. 6, 1047--1074, 2018.

\bibitem{EI86}
A. El Soufi and S. Ilias. 
\emph{Immersions minimales, premi\`{e}re valeur propre du laplacien et volume conforme.}  
Math. Ann. 275 (1986), no. 2, 257–267 (French).

\bibitem{FL20}
P. Freitas and  R. S. Laugesen.
\emph{Two balls maximize the third Neumann eigenvalue in hyperbolic space.} 
Ann. Sc. Norm. Super. Pisa Cl. Sci., to appear. \arxiv{2009.09980}.


\bibitem{GL19} A. Girouard and R. S. Laugesen.
\emph{Robin spectrum: two disks maximize the third eigenvalue.} 
Indiana Univ. Math. J. 70 (2021), no. 6, 2711–2742.

\bibitem{GKL20} 
A. Girouard, M. Karpukhin and J. Lagac\'{e}. 
\emph{Continuity of eigenvalues and shape optimisation for Laplace and Steklov problems.} 
Geom. Funct. Anal. 31, 513–561 (2021). 

\bibitem{GNP09} 
A. Girouard, N. Nadirashvili and I. Polterovich. 
\emph{Maximization of the second positive Neumann eigenvalue for planar domains.} 
J. Differential Geom. 83 (2009), no. 3, 637-661. 

\bibitem{H70}
J. Hersch. 
\emph{Quatre propri\'{e}t\'{e}s isop\'{e}rim\'{e}triques de membranes sph\'{e}riques homog\`{e}nes.}
C. R. Acad. Sci. Paris S\'{e}r. A-B 270 (1970), A1645-A1648.

\bibitem{KLO17} 
C. Kao, R. Lai and B. Osting.
\emph{Maximization of Laplace-Beltrami eigenvalues on closed Riemannian surfaces.} 
ESAIM Control Optim. Calc. Var. 23 (2017), 685-720.

\bibitem{K19} 
M. Karpukhin.  
\emph{Index of minimal spheres and isoperimetric eigenvalue inequalities.} 
Invent. Math. 223 (2020), 335-377.

\bibitem{KNPP19} 
M. Karpukhin, N. Nadirashvili, A. Penskoi and I. Polterovich.  
\emph{An isoperimetric inequality for Laplace eigenvalues on the sphere.} 
J. Diff. Geometry. 118 (2021), no. 2, 313-331.

\bibitem{KS20} 
M. Karpukhin and D. L. Stern.  
\emph{Min-max harmonic maps and a new characterization of conformal eigenvalues.} 
 (2020), \arxiv{2004.04086}. 

\bibitem{KNPP20} 
M. Karpukhin, N. Nadirashvili, A. Penskoi and I. Polterovich. 
\emph{Conformally maximal metrics for Laplace eigenvalues on surfaces.}
Surv. Differ. Geom. 24 (2019), 205-256. 


\bibitem{L20b} 
R. S. Laugesen. 
\emph{Well-posedness of Hersch--Szeg\H{o}'s center of mass by hyperbolic energy minimization.} 
 Ann. Math. Qu\'{e}, 45(2) (2021), 363-390.

\bibitem{LY82}
P. Li and S. Yau.
\emph{A new conformal invariant and its applications to the Willmore conjecture and the first eigenvalue of compact surfaces.} 
Invent. Math. 69 (1982), pp. 269-201

\bibitem{N02}
N. Nadirashvili. 
\emph{Isoperimetric inequality for the second eigenvalue of a sphere.} 
J. Differential Geom. 61 (2002), no. 2, 335-340.

\bibitem{NA18}
N. Nadirashvili and A. Penskoi. 
\emph{An isoperimetric inequality for the second non-zero eigenvalue of the Laplacian on the projective plane.} 
Geom. Funct. Anal. 28 (2018) 1368-1393.

\bibitem{NS17}
N. Nadirashvili and Y. Sire. 
\emph{Isoperimetric inequality for the third eigenvalue of the Laplace–Beltrami operator on $\mathbb{S}^2$.} 
J. Differential Geom. 107 (2017), 561-571.

\bibitem{P14}
R. Petrides. 
\emph{Maximization of the second conformal eigenvalue of spheres.} 
Proc. Amer. Math. Soc. 142 (2014), no. 7, 2385-2394.

\bibitem{P17}
R. Petrides. 
\emph{On the existence of metrics which maximize Laplace eigenvalues on surfaces.} 
Internat. Math. Res. Not. 2018 (2018), 4261–4355.

\bibitem{S16}
Stoll, M. 
Harmonic and Subharmonic Function Theory on the Hyperbolic Ball. 
London Mathematical Society Lecture Note Series. Cambridge, Cambridge University Press, 2016.

\bibitem{S54}
G. Szeg\H{o}. 
\emph{Inequalities for certain eigenvalues of a membrane of given area.}  
J. Rational Mech. Anal. 3 (1954), 343-356.

\bibitem{W56}
H. F. Weinberger. 
\emph{An isoperimetric inequality for the $N$-dimensional free membrane problem.} 
J. Rational Mech. Anal. 5 (1956), 633-636.

\end{thebibliography}

\end{document}